\documentclass[11pt,leqno]{amsart}
\usepackage{amsfonts, amsmath, amsthm, amssymb,xspace}
\usepackage[breaklinks=true]{hyperref}
\usepackage{amsmath,amscd}

\theoremstyle{plain}

\newtheorem{theorem}{Theorem}[section]
\newtheorem{lemma}{Lemma}[section]

\theoremstyle{definition}

\newtheorem{remark}{Remark}[section]

\DeclareMathAlphabet{\mathscr}{OT1}{pzc}{m}{it}

\DeclareMathOperator{\spec}{\ensuremath{Spec}}
\DeclareMathOperator{\Gr}{\ensuremath{gr}}

\DeclareMathOperator{\Sym}{\ensuremath{Sym}}

\begin{document}

\title{The inverse Galois problem for Cherednik algebras}
\author{Akaki Tikaradze}
\address{The University of Toledo, Department of Mathematics, Toledo, Ohio, USA}
\email{\tt tikar06@gmail.com}

\begin{abstract} 

Given the spherical subalgebra  $B$ of a rational Cherednik algebra, we aim to classify all finite groups  $\Gamma$
for which there exists a domain $R$ on which $\Gamma$ acts by ring automorphisms, such that $B=R^{\Gamma}.$
We describe such groups in terms of geometry of the center of the reduction of $B$ modulo a large prime.

\end{abstract}
\maketitle

\section{Introduction and main results}

Given a simple domain $B$ over $\mathbb{C}$, it is an interesting and natural problem to classify
finite groups $\Gamma$ for which there exists a domain $R$ on which $\Gamma$ acts via ring automorphisms
such that $B=R^{\Gamma}.$  Given the direct analogy with Galois theory, 
we refer to this question as the inverse Galois problem for $B.$
In \cite{T} we solved this problem
for rings of differential operators on smooth affine varieties. Namely,
 if $D(X)=R^{\Gamma},$
where $X$ is a smooth affine variety over $\mathbb{C}$ and $\Gamma$ is a finite group of $\mathbb{C}$-automorphisms of a domain $R,$
then there exists a smooth affine variety $Y$ such that $R\cong D(Y)$ and $Y\to X$ is a $\Gamma$-Galois etale covering of $X$
[\cite{T}, Theorem 1 ].
It was also shown in [\cite{T}, Theorem 2] that a very generic central quotient of the  enveloping algebra of a  semi-simple Lie algebra
cannot be a nontrivial fixed ring. In this paper we apply the methodology of \cite{T} to the case when $B$ is a (simple)
 spherical subalgebra of a rational Cherednik algebra defined by Etingof and Ginzburg \cite{EG}. Let us recall their definition.

 Let $W$ be a complex reflection group; $\mathfrak{h}$ its reflection representation
and $S\subset W$  the set of all complex reflections.
Let $(\,,\,):\mathfrak{h}\times \mathfrak{h}^*\to \mathbb{C}$ be the natural pairing. Given a reflection $s\in S,$
let $\alpha_s\in \mathfrak{h}^*$ be an eigenvector of $s$ for eigenvalue $1$.
Also, let $\alpha_s^{\vee} \in \mathfrak{h}$ be an eigenvector normalized so that $\alpha_s(\alpha_s^{\vee} )=2.$
Let $c:S\to \mathbb{C}$ be a function invariant with respect to conjugation by $W.$ 
The rational Cherednik algebra $H_{c}$ associated to $(W, \mathfrak{h})$ with parameter
$c$ is defined as the quotient of $\mathbb{C}[W]\ltimes T( \mathfrak{h}\oplus \mathfrak{h^*})$ by the following relations
$$
[x, y]=(y,x)-\sum_{s\in S}c(s)(y,\alpha_s)(\alpha_s^{\vee} , x),\quad [x, x']=0=[y,y']
$$
for all $x, x'\in \mathfrak{h}^*$ and $y, y'\in \mathfrak{h}.$

In this note we are concerned with the spherical subalgebra $B_c$ of a Cherednik algebra $H_c.$
Recall that  $$B_c=eH_ce, e=\frac{1}{|W|}\sum_{g\in W}g.$$
For $c=0$, we have that $B_0=D(\mathfrak{h})^W.$ Algebras $B_c$ can be viewed as filtered quantizations
of the ring of functions on $(\mathfrak{h}\oplus\mathfrak{h}^*)/W.$

 Since $B_c$ is defined over $S=\mathbb{Z}[\frac{1}{l}][c]$, we may define its base change $(B_c)_{\bf{k}}=B_c\otimes_S\bold{k}$
 which we denote by $B_{\bar{c}},$ where $\bar{c}$ is the image of $c$ under the base change map
 $S\to \bf{k}.$

The following theorem is the main result of this paper. It relates the inverse Galois problem for $B_c$
to geometry of the center of reduction of $B_c$ modulo a large prime. 
\begin{theorem}\label{main}

Let $B_c$ be simple. If $B_c=R^{\Gamma}$ for a domain $R$ and a finite group $\Gamma$, then 
there exists a finitely generated ring $S\subset\mathbb{C} $ containing values of $c$ such that the following holds.
For any base change $S\to\bold{k}$ to an algebraically closed field of  positive characteristic
the group $\Gamma$
is a quotient of the etale fundamental group of the smooth locus of
$\spec Z(B_{\bar{c}}).$ 

\end{theorem}

 We apply Theorem \ref{main} to classes of Cherednik algebras for which the center of their reduction modulo
$p>0$ is well-known and (relatively) easy to describe. Namely, we consider two families of  spherical subalgebra
of the rational Cherednik algebras: one  associated to the pair $(S_n, \mathbb{C}^n)$ and a parameter $c\in\mathbb{C},$
the other family of algebras being noncommutative deformations of Kleinian singularities.

\begin{theorem}\label{Sn}

Let $B_c$ be the spherical subalgebra of a rational Cherednik algebra associated with $(S_n \mathbb{C}^n)$ with a parameter $c\in\mathbb{C}.$
Assume that $B_c$ is simple.
If $c$ is irrational then $B_c$ cannot be a fixed ring of a domain under a nontrivial  action of a finite
group of ring automorphisms. For  rational $c$, if $B_c=R^\Gamma$ with finite group $\Gamma$ and domain $R,$
then $\Gamma$ must be a quotient of $S_n.$

\end{theorem}

Next, we consider the case of noncommutative deformations
of the Kleinian singularities of type A \cite{H} (the spherical subalgebras of Cherednik algebras
associated with the pair a cyclic group and its one dimensional representation).These family of algebras is also known as generalized Weyl algebras. 
Let us recall their definition.

Let $v=\prod_{i=1}^n (h-t_i)\in \mathbb{C}[h].$
Then the algebra $A(v)$ is generated by $x, y, h$ subject to the relations
$$xy=v(h),\quad yx=v(h-1),\quad hx=x(h+1),\quad hy=y(h-1).$$
Recall that if $v=\prod_{i}(h+\frac{i}{n}),$ then $A(v)$ can be identified with the fixed ring
of the Weyl algebra $W_1(\mathbb{C})$ under the natural action of the cyclic group of order $n.$
On the other hand, when $n=2$ algebras $A(v)$ correspond to central quotients of $U(\mathfrak{sl}_2).$
It was shown in \cite{S} that a countable family of primitive quotients
of $U(\mathfrak{sl}_2)$ can be realized as $\mathbb{Z}/2\mathbb{Z}$-fixed rings
of algebras of differential operators on certain (singular) algebraic curves.
\begin{theorem}\label{typeA}
Let $A(v)$ be simple. If $A(v)=R^{\Gamma}$ with $R$ domain and $\Gamma$ a finite group,
then $\Gamma$ must be a quotient of $\mathbb{Z}/n\mathbb{Z}.$ If in addition $t_i-t_j\notin\mathbb{Q}$ for some $i, j,$
then $|\Gamma|<n.$

\end{theorem}





\section{proofs}

We start by recalling couple of very basic properties of the spherical subalgebras of rational Cherednik algebras.
Namely the PBW property and the Dunkl isomorphism.

The crucial PBW property of $H_c, B_c,$ implies that if we equip $H_c, B_c,$
with an algebra filtration by putting 
$$deg(\mathfrak{h})=1,\quad deg(\mathfrak{h^*})=0,\quad deg(W)=0,$$ then
$$\Gr(H_c)=\mathbb{C}[W]\ltimes \Sym(\mathfrak{h}\oplus \mathfrak{h^*}), \quad \Gr(B_c)=\Sym(\mathfrak{h}\oplus \mathfrak{h^*})^{W}.$$
Recall that since for any nonzero $f\in \Sym(\mathfrak{h}^*), \text{ad}(f)=[f, -]$ acts locally nilpotently on $H_c,$
we may consider the localization $H_c[f^{-1}]$ (and $B_c[f^{-1}]$ for $f\in  \mathbb{C}[\mathfrak{h}]^W$).
Then we have the induced filtration on $B_c[f^{-1}]$ and 
$$\Gr(B_c[f^{-1}])=\Sym(\mathfrak{h}\oplus \mathfrak{h^*})^{W}_f.$$

Set $\mathfrak{h}^{reg}=\lbrace x\in \mathfrak{h}, (x, \alpha)\neq 0, \alpha\in S\rbrace.$
Let $\delta \in \mathbb{C}[h]^W$ be the defining function of  $\mathfrak{h}\setminus\mathfrak{h}^{reg}.$ 
 Recall that via the Dunkl embedding we have an isomorphism 
 $$B_c[\delta^{-1}]\cong D(\mathfrak{h}^{\text{reg}}).$$

\begin{proof}[Proof of Theorem\ref{main}]

We denote $B_c$ by $B$ throughout the proof. Since $B$ is a simple Noetherian ring  and  $Z(B)=\mathbb{C},$ it follows
 from the standard facts about fixed rings \cite{M} that $B$ is Morita equivalent to the skew ring $\mathbb{C}[\Gamma]\ltimes R$ (see [\cite{T}, Lemma 4]).
Now, there exists a large enough finitely generated ring $S\subset\mathbb{C}$,
and models of $B, R$ over $S$, to be denoted by $B_S, R_S,$
 so that $B_S$ is Morita equivalent
to $S[\Gamma]\ltimes R_S.$ In particular, $R_S$ is a projective left (and right) $B_S$-module. So for large enough $p\gg 0$ and a base change $S\to \bf{k}$ to an algebraically
closed field of characteristic $p$, we have that $B_{\bf{k}}$ is Morita equivalent to
$\bold{k}[\Gamma]\ltimes R_{\bf{k}}.$
  
   It is well-known that $B_{\bold{k}}$ is finite over its center, more specifically [\cite{BFG}, Theorem 9.1.1]
   $$\Gr Z(B_\bold{k})=\Gr(B_{\bold{k}})^p.$$ 
 Let $f\in Z(B_{\bold{k}})$ be a nonzero element that vanishes on the singular locus of
$\spec(Z(B_{\bold{k}})).$ As the smooth and the Azumaya loci of $\spec(Z(B_{\bold{k}}))$ coincide [\cite{BC}, Theorem ], we get that
$(B_{\bold{k}})_f$ is an Azumaya algebra over $Z(B_{\bold{k}})_f$ and $(B_{\bold{k}})_f$
is Morita equivalent to $\bold{k}[\Gamma]\ltimes (R_{\bf{k}})_f.$
Then just as in [\cite{T}, Proposition 1], we can conclude that $\spec Z(R_{\bold{k}})_f\to \spec Z(B_k)_f$ is
a $\Gamma$-Galois etale covering and
$$(R_{\bold{k}})_f=B_{\bold{k}}\otimes_{Z(B_{\bold{k}})} Z(R)_f.$$
Therefore, if $U$ denotes the smooth locus of $\spec(Z(B_{\bold{k}})),$ and $Y$ denotes
the preimage of $U$ under the projection $\spec(Z(R_{\bold{k}}))\to \spec Z(B_{\bold{k}}),$
then $Y\to U$ is $\Gamma$-Galois covering.
In particular, for any $g\in \bold{k}[\mathfrak{h}]$, $\text{ad}(g)$ acts locally nilpotently on $(R_{\bold{k}})_f.$ Which implies that $\text{ad}(g)$ acts locally
nilpotently on $R_{\bold{k}}$ as $R_{\bold{k}}$ is $Z(B_{\bold{k}})$-torsion free (since $R_{\bold{k}}$ is projective over $B_{\bold{k}}$).
It follows that if $R_{\bold{k}}$ is a domain, then $\Gamma$ is a quotient of the etale fundamental group
of the smooth locus of $\spec(Z(B_{\bold{k}}))=U.$ Thus, all it remains to show is that
$R_{\bold{k}}$ is a domain.

Next we argue that $ad(\delta)$ acts locally nilpotently on $R.$ Indeed, put $B'=B\otimes B^{op}$ and $f=\delta\otimes1-1\otimes\delta.$
So, $B'$ is a spherical subalgebra of a Cherednik algebra associated to $(W\times W, \mathfrak{h}\oplus \mathfrak{h}).$
We can view $R$ as a left $B'$-module.
Recall that we have the filtration
on $B'_S[f^{-1}]$ so that $\Gr (B'[f^{-1}])$ is a finitely generated commutative $S$-algebra. Equipping
$R_S[f^{-1}]$ with a compatible filtration gives that $\Gr(R_S[f^{-1}])$ is a finitely generated $\Gr (B'[f^{-1}])$-module, so by the 
generic flatness theorem there is a localization $S'$ of $S$ so that $\Gr(R_S'[f^{-1}])$ and hence
$R_{S'}[f^{-1}]$ is a free $S'$-module. On the other hand, since $R_{\bold{k}}[f^{-1}]= 0$ for all base changes
$S'\to\bold{k}$ for $\text{char}(\bold{k})\gg 0$, we conclude that $R[f^{-1}]=0.$
Therefore the action of $ad(\delta)$ on $R$ is locally nilpotent.

Let $R'=R[\delta^{-1}].$  Since $R'^{\Gamma}=D(\mathfrak{h}^{reg}),$
it follows from [\cite{T}, Theorem1] that $R'\cong D(Y)$ for some smooth affine variety $Y.$
Hence $R'_{\bf{k}}$ is a domain for $\text{char}(\bold{k})\gg 0$, as desired.

\end{proof}

To use Theorem \ref{main}, we need to know the $p'$-part of the etale fundamental group
of the smooth locus of $\spec(Z(B_{\bar{c}})).$ For this purpose we utilize the following.

\begin{remark}\label{lifting}
Let $X$ be a complete smooth variety over an algebraically closed field $\bf{k}$ of characteristic
$p,$ and $U\subset X$ be an open subset such that $X\setminus U$ is a divisor with normal crossings
in $X.$ Let $\tilde{X}$ be a complete smooth lift of $X$ over $W(\bold{k})$ ($W(\bold{k})$ is the ring of Witt vectors over
$\bold{k}$), $\tilde{U}\subset \tilde{X}$
be an open subset lifting $U$, such that $\tilde{X}\setminus\tilde{U}$ is a divisor with normal crossings over $W(\bold{k})$.
Then any $p'$-degree Galois covering of $U$ admits a lift to a Galois covering of $\tilde{U}$ [\cite{LO}, Corollary A.12], which yields
that any $p'$-quotient of the etale fundamental group of $U$ must be a quotient of the fundamental group of
$U_{\mathbb{C}}.$
 
\end{remark}

We need the following corollary of the Chebotarev density theorem.
It contains slightly more than [\cite{VWW}, Theorem 1.1]. We present a short proof for a reader's convenience.

\begin{lemma}\label{cheb}

Let $S$ be a finitely generated domain containing $\mathbb{Z}$ and $c\in S.$
Then there are infinitely many primes $p$ and ring homomorphisms $\phi_p:S\to F_p.$
If $c\notin \mathbb{Q}$ then there exist infinitely many primes $p$ and homomorphisms
$\phi_p:S\to F_q$, so that $\phi_p(c)\notin F_p$ and $q$ is a power of $p.$
\end{lemma}
\begin{proof}
By the Noether normalization theorem, there exists $l\in \mathbb{N}$ and
algebraically independent $x_1,\cdots, x_n\in S_l$ so that
$S_l$ is integral over $S_l[x_1,\cdots, x_n].$ Let $I$ be a prime ideal in $S_l$ lying over $(x_1,\cdots, x_n)$
(such ideal exists since $\spec(S)\to \spec(S_l[x_1,\cdots, x_n])$ is surjective by the going-up theorem).
 So, $S_l/I=R$ is an integral domain finite over $\mathbb{Z}_l.$
Let $S'$ be the integral closure of $\mathbb{Z}$ in $R.$ Then $R=S'_l.$ 
Thus suffices to show that there exists a homomorphism $\phi_p:S'\to F_p$ for infinitely many $p$. This
 is a consequence of the Chebotarev density theorem.

  We have that the image of the map $\spec(S)\to \spec \mathbb{Z}[c_1]$ contains a nonempty open subset.
If $c_1$ is algebraic, then all but finitely many prime ideals in $\mathbb{Z}[c_1]$ lift to $S.$
By the Chebotarev denisity theorem there are infinitely many primes $I\subset \mathbb{Z}[c_1]$
such that the image of $c_1$ in the quotient  $\mathbb{Z}[c_1]/I\cong F_q$ does not belong to $F_p.$
Let $I'\in\spec(S)$ be a lift of $I.$
Now any homomorphism $S/I'\to \bar{F_p}$ lifting $\mathbb{Z}[c_1]/I\to F_q$ will do.
Finally, let $c_1$ be transcendental. Let $f\in \mathbb{Z}[c_1]$ be such that $\spec (\mathbb{Z}[c_1]_f)$
lifts to $\spec(S).$ Thus it suffices to show that there are infinitely many primes $p$ for which there exists
$t\in\mathbb{Z}[c_1]$ such that $f\notin (p, t)$ and $\mathbb{Z}[c_1]/(p, t)=F_q$ for $q>p.$ 
For this purpose we can take any $p$ that does not divide $f$, then take a nonlinear irreducible $\bar{t}\in F_p[c_1]$
that does not divide $f\mod p.$ Then let $t$ be any lift of $\bar{t}.$
\end{proof}

\begin{remark}

Given a Cherednik algebra $H_c$ associated with an arbitrary pair $(W, \mathfrak{h})$, we expect that
there is a base change to a characteristic $p$ field $\bold{k}$ for infinitely many
values of $p,$ such that 
$$\spec(Z(H_{\bar{c}}))=\spec(Z(B_{\bar{c}}))=(\mathfrak{h}_{\bold{k}}\oplus\mathfrak{h^*}_{\bold{k}})/W.$$
This in view of Theorem \ref{main} would imply that if $B_c=R^{\Gamma},$ where $B_c$ is simple
and $R$ is a domain, then $\Gamma$ must be a quotient of $W.$

\end{remark}

For the proof of Theorem \ref{Sn} we need to recall the definition of the $n$-th Calogero-Moser space. Consider the following
subscheme of pairs of $n$-by-$n$ matrices over $\mathbb{C}$

$$X=\lbrace (A, B)|\quad \text{rank}([A, B]+\text{Id}_n)=1\rbrace.$$
It is known that $PGL_n(\mathbb{C})$ acts freely on $X$ by conjugation, and the $n$-th Calogero Moser space,
denoted by $\text{CM}_n,$ is defined as the quotient 
$$X//PGL_n(\mathbb{C})=\text{CM}_n.$$
It is well-known that $\text{CM}_n$ is a smooth, affine variety over $\mathbb{C}$ \cite{W}.
In the following proof we also need that the Calogero-Moser spaces are simply connected.
This follows from the fact that the $n$-th Calogero-Moser space is homeomorphic to
the Hilbert scheme of $n$-points on the plane which is known to be simply connected based on its cell decomposition.

\begin{proof}[Proof of Theorem \ref{Sn}]

If $c$ is rational then 
after a base change to a field $\bold{k}$ of characteristic $p,$ we have that \cite{BFG}
$$\spec(Z(B_{\bar{c}}))=(\mathfrak{h}\oplus \mathfrak{h}^*)/S_n.$$ 
  Hence using Remark \ref{lifting}, the  $p'$-etale fundamentale group
of the smooth locus of $\spec(Z(B_{\bar{c}}))$ is $S_n.$ Let  $c$ be irrational. By Lemma
\ref{cheb} for any finitely generate subring $S\subset \mathbb{C},$ there are infinitely many primes $p$ and algebraically close fields $\bf{k}$ of characteristic $p$
with a base change $S\to\bold{k}$, such that $\bar{c}\notin F_p.$ Then as explained in \cite{BFG}, we have 
$$\spec Z(B_{\bar{c}})\cong (\text{CM}_{n})_{\bold{k}}.$$
Since $(\text{CM}_{n})_{\bold{k}}$  admits a smooth
simply connected lift to characteristic 0 (namely $\text{CM}_n$), the desired assertion follows.

\end{proof}

\begin{proof}[Proof of Theorem \ref{typeA}]
Let $S\subset \mathbb{C}$ be as in the conclusion of Theorem \ref{main}.
Denote by $\bar{t}_1,\cdots, \bar{t}_n$ images of $t_1,\cdots, t_n$ after a base change $S\to \bf{k}$ to an algebraically closed field of characteristic $p.$
The center of $A(\bar{v})=A(v)\otimes_S\bold{k}$ is known to be generated by $$x_p=x^p, y_p=y^p,h_p= h^p-h$$
subject to the following relation \cite{BC}
$$x_py_p=\prod_{i=1}^n(h_p-(\bar{t_i}^p-\bar{t_i})).$$
After reordering if necessary, 
let $\bar{t_1},\cdots,\bar{t_k}$ be representatives of all distinct cosets of $\bar{t_i}+F_p$ with multiplicities
$a_1,\cdots, a_k.$ So 
$$x_py_p=\prod_{i=1}^k(h^p-h-(\bar{t_i}^p-\bar{t_i}))^{a_i}.$$
Hence the singular locus of $\spec Z(A(\bar{v}))$ is
$$\lbrace x_p=0=y_p, h_p=\bar{t_i}^p-\bar{t_i}, a_i>1\rbrace.$$
Let $\alpha_i, 1\leq i \leq k$ be lifts of $\bar{t_i}^p-\bar{t_i}$ in $W(\bf{k}).$
Let $U$ be the smooth locus of $$\spec W(\bold{k})[a, b, c]/(ab-\prod_{i\leq k}(c-\alpha_i)^{a_i}).$$
Then $U$ is a lift of the smooth locus of $\spec Z(A(\bar{v}))$ over $W(\bf{k})$ and
the fundamental group of $U_{\mathbb{C}}$ is 
 $\mathbb{Z}/gcd(a_1,\cdots, a_k)\mathbb{Z}.$
 Using Lemma \ref{cheb}, there are infinitely many $p$ for which there exists a base change
 such that $\bar{t_i}\in F_p$ for all $i.$ Hence, $\Gamma$ must be a quotient of  $\mathbb{Z}/n\mathbb{Z}.$
  Let $t_1-t_2\notin \mathbb{Q}.$ Then again using Lemma \ref{cheb}, there exists a base change so that
  $\bar{t_1}^p-\bar{t_1}\neq \bar{t_2}^p-\bar{t_2}.$ Therefore, in the corresponding partition $(a_1,\cdots, a_k)$
  all numbers are less than $n.$ Hence, $gcd(a_1,\cdots, a_k)=d<n.$ Thus $\Gamma$ must be a quotient of 
  $\mathbb{Z}/d\mathbb{Z}.$
\end{proof}




\begin{remark}
It seems natural to expect that Theorem \ref{main} should hold for general filtered quantizations.
Namely, given a simple $\mathbb{C}$-domain $A$ that can be equipped with an ascending filtration
such that the corresponding associated graded algebra is a finitely generated commutative $\mathbb{C}$-domain.
In this setting, if $A=R^G,$ then it seems reasonable to expect that $G$ must appear as a quotient of the
etale fundamental group of the Azumaya locus of $\spec(Z(A_{\bold{k}}))$ for char$(\bold{k})\gg 0.$
The proof of Theorem \ref{main} can easily be adapted to prove this provided that
$R_{\bf{k}}$ is a domain. It also follows from the proof that $R$ must be a Harish-Chandra bimodule over $A.$

\end{remark}

\end{document}